\documentclass[10pt,reqno]{amsart}

\usepackage{a4wide}
\usepackage{dsfont}          
\usepackage{amssymb,amsmath}
\usepackage{mathrsfs}
\usepackage{array}
\usepackage[utf8]{inputenc}
\usepackage[T1]{fontenc}
\usepackage[all]{xy}
\usepackage{enumerate}
\usepackage{comment}

\newtheorem{proposition}{Proposition}[section]
  \newtheorem{theorem}[proposition]{Theorem}
  \newtheorem{fact}[proposition]{Fact}
\theoremstyle{remark}
  \newtheorem{definition}[proposition]{Definition}
  \newtheorem{defprop}[proposition]{Definition-proposition}
  \newtheorem{example}[proposition]{Example}
  \newtheorem{remark}[proposition]{Remark}

\newcommand{\cst}{\ifmmode\mathrm{C}^*\else{$\mathrm{C}^*$}\fi}
\newcommand{\st}{\;\vline\;}
\newcommand{\RR}{\mathbb{R}}
\newcommand{\NN}{\mathbb{N}}
\newcommand{\ph}{\varphi}
\newcommand{\eps}{\varepsilon}
\newcommand{\CC}{\mathbb{C}}
\newcommand{\tens}{\otimes}

\newcommand{\bh}{\boldsymbol{h}}
\newcommand{\id}{\mathrm{id}}
\newcommand{\comp}{\circ}
\newcommand{\I}{\mathds{1}}
\newcommand{\GG}{\mathbb{G}}
\newcommand{\QG}{\mathbb{G}}
\newcommand{\sA}{\mathsf{A}}
\newcommand{\sB}{\mathsf{B}}
\newcommand{\cH}{\mathscr{H}}
\newcommand{\cF}{\mathscr{F}}
\newcommand{\hh}[1]{\widehat{#1}}
\newcommand{\ww}{\mathrm{W}}
\newcommand{\vv}{\mathrm{V}}

\newcommand{\tensmin}{\tens_\textup{\tiny{min}}}

\DeclareMathOperator{\Pol}{Pol}
\DeclareMathOperator{\C}{C}
\DeclareMathOperator{\B}{B}
\DeclareMathOperator{\Mor}{Mor}
\DeclareMathOperator{\M}{M}
\DeclareMathOperator{\Tr}{Tr}
\DeclareMathOperator{\Irr}{Irr}
\DeclareMathOperator{\cc0}{c_0}
\DeclareMathOperator{\Lin}{\overline{Lin}}
\DeclareMathOperator{\Prob}{Prob}
\DeclareMathOperator{\Rep}{Rep}
\DeclareMathOperator{\Hom}{Hom}

\usepackage{amsmath,accents}
\DeclareMathSymbol{\widehatsym}{\mathord}{largesymbols}{"62}
\newcommand\lowerwidehatsym{%
  \text{\smash{\raisebox{-1.3ex}{%
    $\widehatsym$}}}}
\newcommand\fixwidehat[1]{%
  \mathchoice
    {\accentset{\displaystyle\lowerwidehatsym}{#1}}
    {\accentset{\textstyle\lowerwidehatsym}{#1}}
    {\accentset{\scriptstyle\lowerwidehatsym}{#1}}
    {\accentset{\scriptscriptstyle\lowerwidehatsym}{#1}}
}
\newcommand{\hcA}{\!\fixwidehat{\,\mathscr{A}}}

\numberwithin{equation}{section}

\author{Uwe Franz}
\address{Laboratoire de Math\'ematiques de Besançon, Universit\'e de Franche-Comt\'e}
\email{uwe.franz@univ-fcomte.fr}

\author{Adam Skalski}
\address{Institute of Mathematics of the Polish Academy of Sciences, Warsaw}
\email{a.skalski@impan.pl}

\author{Piotr M.~So{\l}tan}
\address{Department of Mathematical Methods in Physics, Faculty of Physics, University of Warsaw, Poland}
\email{piotr.soltan@fuw.edu.pl}

\title[Compact and discrete quantum groups]{Introduction to compact and discrete quantum groups}

\begin{document}

\begin{abstract}
These are notes from introductory lectures at the graduate school ``Topological Quantum Groups'' in Będlewo (June 28--July 11, 2015). The notes present the passage from Hopf algebras to compact quantum groups and sketch the notion of discrete quantum groups viewed as duals of compact quantum groups.
\end{abstract}

\maketitle

\section{From Hopf algebras to compact quantum groups}

In the first chapter we motivate and introduce the notion of Hopf *-algebras in the purely algebraic setting and discuss first examples of compact quantum groups. All the vector spaces will be vector spaces over $\CC$, Hilbert space scalar products will be linear on the right.

\subsection{Algebraic tensor product}

Let $V_1,\dotsc,V_n$ be vector spaces (over $\CC$). There exists a vector space $V_1\tens\dotsm\tens{V_n}$ and a multi-linear map $\iota\colon{V_1}\times\dotsm\times{V_n}\to{V_1}\tens\dotsm\tens{V_n}$ such that for any vector space $W$ and any multi-linear map $f\colon{V_1}\times\dotsm\times{V_n}\to{W}$ there exists a unique linear map $\hh{f}\colon{V_1}\tens\dotsm\tens{V_n}\to{W}$ such that the diagram
\[
\xymatrix{
{V_1}\times\dotsm\times{V_n}\ar[r]^\iota\ar[d]_(.48)f&V_1\tens\dotsm\tens{V_n}\ar[dl]^{\hh{f}}\\
W
}
\]
commutes. We call $V_1\tens\dotsm\tens{V_n}$ the (algebraic) \emph{tensor product} of vector space $V_1,\dotsc,V_n$. It is not difficult to see that the vector space $V_1\tens\dotsm\tens{V_n}$ is spanned by the image of the multi-linear map $\iota$. Given $v_i\in{V_i}$ ($i=1,\dotsc,n$) we denote $\iota(v_1,\dotsc,v_n)$ by the symbol $v_1\tens\dotsm\tens{v_n}$ and such elements are called \emph{simple tensors}. Any $v\in{V_1}\tens\dotsm\tens{V_n}$ is a linear combination of simple tensors.

The tensor product, denoted by $\tens$, is a functor: for linear maps $f_i\colon{V_i}\to{W_i}$, $i=1,\dotsc,n$, there exists
$f_1\tens\dotsm\tens{f_n}\colon{V_1}\tens\dotsm\tens{V_n}\to{W_1}\tens\dotsm\tens{W_n}$ such that the diagram
\[
\xymatrix{
{V_1}\times\dotsm\times{V_n}\ar[r]^\iota\ar[d]_(.48){f_1\times\dotsm\times{f_n}}&V_1\tens\dotsm\tens{V_n}\ar[d]^(.48){f_1\tens\dotsm\tens{f_n}}\\
{W_1}\times\dotsm\times{W_n}\ar[r]_\iota&W_1\tens\dotsm\tens{W_n}\\
}
\]
commutes ($f_1\times\dotsm\times{f_n}$ above denotes the usual Cartesian product of linear maps).

\begin{remark}
The category of vector spaces becomes in this way a \emph{monoidal category}.
\end{remark}

\subsection{Algebras and coalgebras}\label{algCoalg}

\begin{definition}
An \emph{algebra} (more precisely a \emph{unital associative algebra}) is a triple $(A,m,e)$ with $A$ a vector space, $m\colon{A}\tens{A}\to{A}$ and $e\colon\CC\to{A}$ linear maps, such that the diagrams
\[
\xymatrix{
A\tens{A}\tens{A}\ar[r]^(0.57){\id\tens{m}}\ar[d]_(.48){m\tens\id}&A\tens{A}\ar[d]^(.48)m\\
A\tens{A}\ar[r]_(.57)m&A
}
\]
and
\[
\xymatrix{
A\tens\CC\ar[d]_(.48){\id\tens{e}}&A\ar[l]_(.37){\cong}\ar[d]^(.48){\id}\ar[r]^(.37){\cong}&\CC\tens{A}\ar[d]^(.48){e\tens\id}\\
A\tens{A}\ar[r]_(.55)m&A&A\tens{A}\ar[l]^(.55)m
}
\]
commute.
\end{definition}

Let us note that if $(A_i,m_i,e_i)$ is an algebra for $i=1,2$ then the tensor product $A_1\tens{A_2}$ carries a natural structure of an algebra. Indeed, multiplication on $A_1\tens{A_2}$ is given by $m=(m_1\tens{m_2})\comp(\id_{A_1}\tens\tau_{A_1,A_2}\tens\id_{A_2})$, where $\tau_{A_1,A_2}\colon{A_1}\tens{A_2}\to{A_2}\tens{A_1}$ is the flip map, i.e.~the unique linear map taking  each simple tensor $a_1\tens{a_2}$ to $a_2\tens{a_1}$, while the unit of $A_1\tens{A_2}$ is $e=e_1\tens{e_2}$ (we canonically identify $\CC\tens\CC$ with $\CC$).

Furthermore the unit maps $e_1$ and $e_2$ provide ways to embed $A_1$ and $A_2$ into $A_1\tens{A_2}$: we have $\iota_1\colon{A_1}\cong{A_1}\tens\CC\xrightarrow{\id\tens{e_2}}A_1\tens{A_2}$ and $\iota_2\colon{A_2}\cong\CC\tens{A_2}\xrightarrow{e_1\tens\id}A_1\tens{A_2}$. Clearly denoting by $\I_{A_1}$ and $\I_{A_2}$ the elements $e_1(1)$ and $e_2(1)$ respectively we see that the embeddings are simply
\[
A_1\ni{a_1}\longmapsto{a_1}\tens\I_{A_2}\in{A_1}\tens{A_2}\quad\text{and}\quad
A_2\ni{a_2}\longmapsto\I_{A_1}\tens{a_2}\in{A_1}\tens{A_2}.
\]
Similarly if $A_1,A_2$ and $A_3$ are algebras then we have embeddings
\[
\begin{split}
\iota_{12}\colon{A_1}\tens{A_2}\ni{a_1}\tens{a_2}&\longmapsto{a_1}\tens{a_2}\tens\I_{A_3}\in{A_1}\tens{A_2}\tens{A_3},\\
\iota_{23}\colon{A_2}\tens{A_3}\ni{a_2}\tens{a_3}&\longmapsto\I_{A_1}\tens{a_2}\tens{a_3}\in{A_1}\tens{A_2}\tens{A_3},\\
\iota_{13}\colon{A_1}\tens{A_3}\ni{a_1}\tens{a_3}&\longmapsto{a_1}\tens\I_{A_2}\tens{a_3}\in{A_1}\tens{A_2}\tens{A_3}.
\end{split}
\]
Now given $X\in{A_i}\tens{A_j}$ (with $1\leq{i}<j\leq{3}$) we write $X_{ij}$ for the image of $X$ under $\iota_{ij}$. This is the \emph{leg numbering notation} and it is used extensively in many texts on quantum groups. This easily generalizes to multiple tensor products of algebras.

Let us turn now to the definition of a coalgebra. We obtain it by ``dualizing'' i.e.~reverting all arrows in the definition of an algebra:

\begin{definition}\label{coAlg}
A \emph{coalgebra} (more precisely a \emph{counital coassociative coalgebra}) is a triple $(C,\Delta,\eps)$ with $C$ a vector space, $\Delta\colon{C}\to{C}\tens{C}$ and $\eps\colon{C}\to\CC$ linear maps, such that the diagrams
\[
\xymatrix{
C\tens{C}\tens{C}\ar@{<-}[r]^(0.6){\id\tens\Delta}\ar@{<-}[d]_(.48){\Delta\tens\id}&C\tens{C}\ar@{<-}[d]^(.48)\Delta\\
C\tens{C}\ar@{<-}[r]_(.6)\Delta&C
}
\]
and
\[
\xymatrix{
C\tens\CC\ar@{<-}[d]_(.5){\id\tens\eps}&C\ar@{<-}[l]_(.42){\cong}\ar@{<-}[d]^{\id}\ar@{<-}[r]^(.42){\cong}&\CC\tens{C}\ar@{<-}[d]^(.5){\eps\tens\id}\\
C\tens{C}\ar@{<-}[r]_(.65)\Delta&C&C\tens{C}\ar@{<-}[l]^(.65)\Delta
}
\]
commute. The maps $\Delta$ and $\eps$ are called respectively the \emph{comultiplication} (or \emph{coproduct}) and the \emph{counit}.
\end{definition}

The notions of morphisms of algebras and coalgebras are also related by duality:

\begin{definition}
A map $f\colon{A_1}\to{A_2}$ ($f\colon{C_2}\to{C_1}$ respectively) is a \emph{morphism of (co-)algebras} if the diagrams
\[
\xymatrix{
A_1\tens{A_1}\ar[r]^(.48){f\tens{f}}\ar[d]_(.48)m&A_2\tens{A_2}\ar[d]^(.48)m\\
A_1\ar[r]_(.45)f&A_2
}
\qquad\qquad\text{(or}\quad
\xymatrix{
C_1\tens{C_1}\ar@{<-}[r]^(.55){f\tens{f}}\ar@{<-}[d]_(.52)\Delta&C_2\tens{C_2}\ar@{<-}[d]^(.52)\Delta\\
C_1\ar@{<-}[r]_(.55)f&C_2
}\quad\text{resp.)}
\]
and
\[
\xymatrix{
\CC\ar[r]^{\cong}\ar[d]_e&\CC\ar[d]^e\\
A_1\ar[r]_(.45)f&A_2
}
\qquad\qquad\text{(or}\quad
\xymatrix{
\CC\ar@{<-}[r]^(.52){\cong}\ar@{<-}[d]_\eps&\CC\ar@{<-}[d]^\eps\\
C_1\ar@{<-}[r]_(.55)f&C_2
}\quad\text{resp.)}
\]
commute.
\end{definition}

Given a coalgebra $(C,\Delta,\eps)$ and an element $c\in{C}$ the image of $c$ under $\Delta$ can be expressed as
\[
\Delta(c) = \sum_{i=1}^N{c_{1,i}}\tens{c_{2,i}}.
\]
However it has become customary to use the \emph{Sweedler notation}, i.e.~write
\[
\Delta(c)=\sum_{(c)}{c_{(1)}}\tens{c_{(2)}}
\]
and similarly
\[
(\Delta\tens\id)\Delta(c)=\sum_{(c)}{c_{(1)}}\tens{c_{(2)}}\tens{c_{(3)}}
\]
etc.. The notation is sometimes very useful when performing computations. As an example let us rewrite the second diagram from Definition \ref{coAlg} using Sweedler notation:
\[
\sum_{(c)}{c_{(1)}}\eps(c_{(2)})=c=\sum_{(c)}\eps(c_{(1)})c_{(2)},\qquad{c}\in{C}.
\]

\subsection{Bialgebras}

As in Section \ref{algCoalg}, let us denote by $\tau_{V,W}$ the flip map $V\tens{W}\ni{v}\tens{w}\mapsto{w}\tens{v}\in{W}\tens{V}$

\begin{proposition}
If $(A,m,e)$ is an algebra then $(A\tens{A},m_{\tens},e\tens{e})$ with
\[
m_\tens=(m\tens{m})\comp(\id\tens\tau_{A,A}\tens\id)
\]
is also an algebra.
\end{proposition}

\begin{proposition}
If $(C,\Delta,\eps)$ is a coalgebra then $(C\tens{C},\Delta_\tens,\eps\tens\eps)$ with
\[
\Delta_\tens=(\id\tens\tau_{C,C}\tens\id)\comp(\Delta\tens\Delta)
\]
is also a coalgebra.
\end{proposition}

\begin{remark}
$(\CC,\id,\id)$ is an algebra and a coalgebra (to be precise, we use here the canonical identification of $\CC \tens \CC$ with $\CC$).
\end{remark}

\begin{defprop}\label{bialg}
$(B,m,e,\Delta,\eps)$ is a \emph{bialgebra} if
\begin{itemize}
\item $(B,m,e)$ is an algebra,
\item $(B,\Delta,\eps)$ is a coalgebra,
\item the following equivalent conditions are satisfied
\begin{itemize}
\item $\Delta$ and $\eps$ are morphisms of algebras,
\item $m$ and $e$ are morphisms of coalgebras.
\end{itemize}
\end{itemize}
\end{defprop}

Let us note that the compatibility conditions of Definition-Proposition \ref{bialg} mean $\Delta\comp{m}=m_\tens\comp(\Delta\tens\Delta)$ i.e.~the diagram
\[
\xymatrix{
&B\tens{B}\ar[ldd]_{m}\ar[rd]^{\Delta\tens\Delta}\\
&&B\tens{B}\tens{B}\tens{B}\ar[dd]^{\id\tens\tau_{B,B}\tens\id}\\
B\ar[rdd]_\Delta\\
&&B\tens{B}\tens{B}\tens{B}\ar[ld]^{m\tens{m}}\\
&B\tens{B}
}
\]
commutes. This observation constitutes the proof of Definition-Proposition \ref{bialg}.

\subsection{Hopf algebras}

\begin{definition} \label{Defconv}
For an algebra $(A,m,e)$ and a coalgebra $(C,\Delta,\eps)$ we define a multiplication $\star$ called \emph{convolution} on $\Hom(C,A)$ (the space of linear maps from $C$ to $A$) by
\[
f_1\star{f_2}=m\comp(f_1\tens{f_2})\comp\Delta,\qquad{f_1,f_2}\in\Hom(C,A).
\]
\end{definition}

Convolution turns the space $\Hom(C,A)$ into an algebra with unit $e\comp\eps$.

\begin{definition}
A bialgebra $(B,m,e,\Delta,\eps)$ is called a \emph{Hopf algebra} if $\id\colon{B}\to{B}$ is invertible in the algebra $\Hom(B,B)$.
\end{definition}

When the convolution-inverse of $\id$ exists we denote it by $S$. The condition that $S$ is the convolution-inverse of $\id$ is that the diagram
\[
\xymatrix{
B\tens{B}\ar[d]_(.48){S\tens\id}&B\ar[l]_(.38)\Delta\ar[r]^(.38)\Delta\ar[d]_(.48){e\tens\eps}&B\tens{B}\ar[d]^(.48){\id\tens{S}}\\
B\tens{B}\ar[r]_(.55)m&B&B\tens{B}\ar[l]^(.55)m
}
\]
commutes. $S$ is unique (if it exists) and we call it the \emph{antipode}.

\begin{proposition}
$S$ is an algebra and coalgebra anti-homomorphism, i.e.
\[
S\comp{m}=m\comp\tau_{B,B}\comp(S\tens{S}),\qquad\Delta\comp{S}=(S\tens{S})\comp\tau_{B,B}\comp\Delta.
\]
\end{proposition}

\subsection{$*$-Hopf algebras}

\begin{definition}
A \emph{$*$-Hopf algebra} $(H,m,e,\Delta,\eps,S,*)$ is a Hopf algebra $(H,m,e,\Delta,\eps,S)$ equipped with a conjugate linear anti-multiplicative involution $*\colon{H}\to{H}$ such that $\Delta\colon{H}\to{H}\tens{H}$ is a $*$-morphism (the involution on $H\tens{H}$ is $(a\tens{b})^*=a^*\tens{b^*}$).
\end{definition}

\begin{proposition}
\noindent
\begin{enumerate}
\item The counit $\eps$ of a $*$-Hopf algebra is a $*$-homomorphism.
\item The antipode of a $*$-Hopf algebra satisfies $*\comp{S}\comp{*}\comp{S}=\id$.
\end{enumerate}
\end{proposition}

\subsection{Opposites and co-opposites}

If $A=(A,m,e)$ is an algebra, then $A^{{\rm op}}=(A,m^{{\rm op}},e)$ with $m^{{\rm op}}=m\circ \tau_{A,A}$ is also an algebra. Similarly, if $C=(C,\Delta,\varepsilon)$ is a coalgebra, then $C^{{\rm cop}}(C,\Delta^{{\rm cop}},\varepsilon)$ with $\Delta^{{\rm cop}}=\tau_{C,C}\circ \Delta$ is also a coalgebra.

In the same manner, if $(B,m,e,\Delta,\varepsilon)$ is a (*-)bialgebra, we can build three more (*-)bialgebras  $B^{{\rm op}}=(B,m^{{\rm op}},e,\Delta,\varepsilon)$,  $B^{{\rm cop}}=(B,m,e,\Delta^{{\rm cop}},\varepsilon)$, and  $B^{{\rm opcop}}=(B,m^{{\rm op}},e,\Delta^{{\rm cop}},\varepsilon)$. This can be checked by inspection of the corresponding commutative diagrams.

If $H=(H,m,e,\Delta,\varepsilon,S)$ is a Hopf algebra, then we can form its \emph{opposites} and \emph{co-opposites} $H^{{\rm op}}=(H,m^{{\rm op}},e,\Delta,\varepsilon)$, $H^{{\rm cop}}=(H,m,e,\Delta^{{\rm cop}},\varepsilon)$, and  $H^{{\rm opcop}}=(H,m^{{\rm op}},e,\Delta^{{\rm cop}},\varepsilon)$ as bialgebras. It turns out that $S$ is also an antipode for $H^{{\rm opcop}}$, so $H^{{\rm opcop}}$ is again a Hopf algebra. The bialgebras $H^{{\rm op}}$ and $H^{{\rm opcop}}$ however need not have an antipode. If the antipode of $H$ is invertible, then its inverse $S^{-1}$ is an antipode for $H^{{\rm op}}$ and $H^{{\rm opcop}}$.

Since the antipode of a $*$-Hopf algebra $H=(H,m,e,\Delta,\varepsilon,S,*)$ is always invertible (with inverse $S^{-1}=*\circ S\circ *$), we can form the three *-Hopf algebras $H^{{\rm op}}=(H,m^{{\rm op}},e,\Delta,\varepsilon,S^{-1},*)$, $H^{{\rm cop}}=(H,m,e,\Delta^{{\rm cop}},\varepsilon,S,*)$, and  $H^{{\rm opcop}}=(H,m^{{\rm op}},e,\Delta^{{\rm cop}},\varepsilon,S,*)$. Since the antipode is an algebra and coalgebra anti-homomorphism, it is actually a $*$-Hopf algebra isomorphism from $H$ to $H^{{\rm opcop}}$. Similarly, $H^{{\rm op}}$ and $H^{{\rm cop}}$ are isomorphic via the antipode.

\subsection{Examples}

\begin{example}
If $G$ is a group then the group algebra $\CC{G}$ (i.e.\ the vector space spanned by basis elements $\delta_g$, $g \in G$ and equipped with the linear extension of the product $\delta_g \cdot \delta_h = \delta_{gh}$ for $g,h \in G$) is a $*$-Hopf algebra with the coproduct, counit and antipode:
\[
\Delta(g)=g\tens{g},\quad\eps(g)=1,\quad{S}(g)=g^{-1}
\]
for $g\in{G}$ (where we simply write $g \in \CC G$ instead of $\delta_g$). Note that the unit is given by $\delta_e$, where $e$ is the neutral element of $G$.
\end{example}

\begin{example}
If $H$ is a finite dimensional $*$-Hopf algebra. then the dual space $H'$ (the space of all linear maps from $H$ to $\CC$)  is a $*$-Hopf algebra with dual operations:
\[
m_{H'}=\Delta_H',\quad{e}_{H'}=\eps_H',\quad\Delta_{H'}=m_H',\quad\eps_{H'}=e_H',\quad{S_{H'}}=S_H'
\]
and involution $(f^*)(a)=\overline{f\bigl(S(a)^*\bigr)}$ for $f\in{H'}$, $a\in{H}$.
\end{example}

\begin{example}
If $G$ is a finite group then the algebra $\CC^G$of functions on $G$ is a $*$-Hopf algebra with
\[
\Delta(f)(g_1,g_2)=f(g_1g_2),\quad\eps(f)=f(e),\quad{S}(f)(g)=f(g^{-1})
\]
for $f\in\CC^G$, $g_1,g_2,g\in{G}$ and $e$ being the neutral element of $G$ (we also identify $\CC^{G\times{G}}$ with $\CC^G\tens\CC^G$) and involution
\[
\overline{f}(g)=\overline{f(g)},\qquad{f}\in\CC^G,\:g\in{G}.
\]
Let us note that in this case $(\CC{G})'\cong\CC^G$ and $(\CC^G)'=\CC{G}$ as Hopf $*$-algebras.
\end{example}

There are also examples of finite-dimensional $*$-Hopf algebras which do not arise in either of the two ways described above (e.g.\ the algebra associated to the \emph{Kac-Paljutkin quantum group}, which is described for example in \cite{UweRolf}).

We can summarize that the category of finite-dimensional $*$-Hopf algebras has a nice duality theory and includes finite groups in the form of group $*$-algebras and function algebras. To extend this category to include also infinite groups we will now introduce some functional analytic prerequisites.

\subsection{\cst-algebras}

\begin{definition}
A $*$-algebra $A=(A,m,e,*)$ equipped with a vector space norm $\|\,\cdot\,\|$ such that $\bigl(A,\|\,\cdot\,\|\bigr)$ is a Banach space is called a \emph{\cst-algebra} if
\begin{itemize}
\item $\|\,\cdot\,\|$ is \emph{submultiplicative}, i.e.
\[
\|ab\|\leq\|a\|\|b\|,\qquad{a,b}\in{A},
\]
\item $\|\,\cdot\,\|$ satisfies the \emph{\cst-identity}
\[
\|a^*a\|=\|a\|^2,\qquad{a}\in{A}.
\]
\end{itemize}
\end{definition}

\begin{remark}
Our definition of algebras included the existence of a unit, but non-unital \cst-algebras are defined the same way.
\end{remark}

\begin{example}
If $X$ is a compact Hausdorff space then
\[
\C(X)=\bigl\{f\colon{X}\to\CC\st{f}\text{ is continuous}\bigr\}
\]
is a unital \cst-algebra with the norm
\[
\|f\|_\infty=\sup_{x\in{X}}\bigl|f(x)\bigr|.
\]
By a theorem of Gelfand and Naimark, all commutative unital \cst-algebras are of this form (up to isometric $*$-isomorphism). Let us also remark that non-unital commutative \cst-algebras correspond to locally compact spaces (they are all of the form $\C_0(X)$ for a locally compact space $X$, where $\C_0(X)$ denotes the algebra of all continuous complex-valued functions on $X$ vanishing at infinity, equipped with the natural algebraic operations and the supremum norm).
\end{example}

\begin{example}
If $\cH$ is a Hilbert space, then
\[
\B(\cH)=\bigl\{T\colon\cH\to\cH\st{T}\text{ is linear and bounded}\bigr\}
\]
is a unital \cst-algebra with the operator norm and Hermitian conjugation as the $*$-operation. Any norm-closed involutive (closed under conjugation) subalgebra of $\B(\cH)$ is also a \cst-algebra. By a theorem of Gelfand and Naimark, all \cst-algebras are of this form (up to isometric $*$-isomorphism). In general a $^*$-homomorphism from a \cst-algebra $A$ to $\B(\cH)$ for some Hilbert space is called a \emph{representation} of $A$.
\end{example}

To any \emph{state} on a unital \cst-algebra, i.e.\ a positive functional $\omega \in A^*$ such that $\omega(1_A)=1$ one can associate a canonical way a representation $(\pi_\omega, \cH_{\omega})$, via a so-called \emph{GNS} (Gelfand-Naimark-Segal) \emph{construction} (see \cite{Murphy}).

\subsection{Minimal tensor product}

Let $A$ and $B$ be two \cst-algebras. In general $A\tens{B}$ is not a \cst-algebra, if $\tens$ is the (algebra) tensor product.

\begin{definition}
Let
\[
\|c\|_\textrm{\tiny{min}}=\sup_{\rho_A,\rho_B}\biggl\|\sum\rho_A(a_i)\tens\rho_B(b_i)\biggr\|
\]
for $c=\sum{a_i}\tens{b_i}\in{A}\tens{B}$, where the supremum is taken over all representations $(\rho_A,\cH_A)$ and $(\rho_B,\cH_B)$ of $A$ and $B$, and the norm on the right hand side is the operator norm on $\cH_A\tens_2\cH_B$ (the notation $\tens_2$ refers to the completed, Hilbert space tensor product).

The completion
\[
A\tens_\textrm{\tiny{min}}B=\overline{A\tens{B}}^{\,\|\,\cdot\,\|_\textrm{\tiny{min}}}
\]
of $A\tens{B}$ in this norm is a \cst-algebra. It is called the \emph{minimal} (or \emph{spatial}) \emph{tensor product} of $A$ and $B$.
\end{definition}

\begin{example}
For compact spaces $X$ and $Y$ we have a canonical isomorphism
\[
\C(X)\tens_\textrm{\tiny{min}}\C(Y)\cong\C(X\times{Y}).
\]
\end{example}

\subsection{Compact quantum groups}\label{cqgs}

\begin{definition}[Woronowicz]\label{defQG1}
A \emph{compact quantum group} is a pair $\GG=(A,\Delta)$, where $A$ is a unital \cst-algebra,
\[
\Delta\colon{A}\longrightarrow{A}\tens_\textrm{\tiny{min}}A
\]
is a unital $*$-homomorphism such that
\begin{itemize}
\item $\Delta$ is coassociative, i.e.~$(\Delta\tens\id)\comp\Delta)=(\id\tens\Delta)\comp\Delta$,
\item the quantum cancellation rules are satisfied:
\[
\Lin\bigl\{(\I\tens{A})\Delta(A)\bigr\}={A}\tens_\textrm{\tiny{min}}A=\Lin\bigl\{(A\tens\I)\Delta(A)\bigr\}.
\]
\end{itemize}
\end{definition}

A is called the \emph{algebra of “continuous functions”} on $\GG$ and also denoted by $\C(\GG)$.

Informally
a \emph{morphism} of compact quantum groups between compact quantum groups $\GG_1=(A_1,\Delta_1)$ and $\GG_2=(A_2,\Delta_2)$ could be thought of as a unital $*$-homomorphism
$\pi\colon{A_2}\to{A_1}$ such that
\[
\Delta_1\comp\pi=(\pi\tens\pi)\comp\Delta_2.
\]
Note the inversion of arrows in the above! In fact  the actual definition of the morphism between compact quantum groups needs to be slightly modified, to take into account the analytical subtleties. We will return to this in the next chapter.

\subsection{Examples coming from groups}

\begin{example}
A compact group $G$ can be viewed as a compact quantum group with $A=\C(G)$ and
\[
\Delta_G\colon\C(G)\longrightarrow\C(G\times{G})\cong\C(G)\tens_\textrm{\tiny{min}}\C(G)
\]
defined by
\[
\Delta_G(f)(g_1,g_2)=f(g_1g_2),\qquad{f}\in\C(G),\:g_1,g_2\in{G}.
\]
\end{example}

\begin{remark}
A continuous group homomorphism $\ph\colon{G_1}\to{G_2}$ induces a morphism of compact quantum groups
\[
\pi_\ph\colon\C(G_2)\longrightarrow\C(G_1)
\]
by
\[
\pi_\ph(f)=f\comp\ph.
\]
This explains the inversion of arrows in the notion of morphisms suggested above.
\end{remark}

\begin{theorem}
If $\GG=(A,\Delta)$ is a \emph{commutative} compact quantum group (i.e.~$A$ is commutative) then there exists a compact group $G$ such that $\GG$ is isomorphic to $\bigl(\C(G),\Delta_G\bigr)$ i.e.~there exists a $*$-isomorphism
\[
\pi\colon{A}\longrightarrow\C(G)
\]
such that
\[
\Delta_G\comp\pi=(\pi\tens\pi)\comp\Delta
\]
(in other words $\pi$ is an isomorphism of quantum groups).
\end{theorem}

\begin{example}
For a discrete group $\Gamma$ we can turn the (completed) reduced and universal group \cst-algebras $\cst_\textrm{\tiny{r}}(\Gamma)$ and $\cst_\textrm{\tiny{u}}(\Gamma)$ into compact quantum groups, denoted by $\hh{\Gamma}$, if we set
\[
\Delta(\gamma)=\gamma\tens\gamma
\]
for $\gamma\in\Gamma$.
\end{example}

\begin{theorem}
If $\GG=(A,\Delta)$ is a cocommutative compact quantum group (i.e.~$\tau_{A,A}\comp\Delta=\Delta$) then there exists a discrete group $\Gamma$ and surjective
unital $^*$-homomorphisms
\[
\xymatrix{
{\cst_\textrm{\tiny{u}}}(\Gamma)\ar[r]^(.6){\pi_1}&A\ar[r]^(.35){\pi_2}&{\cst_\textrm{\tiny{r}}}(\Gamma)
}
\]
intertwining the respective coproducts.
\end{theorem}

\subsection{Non-commutative and non-cocommutative examples}

Now let $A$ be a \cst-algebra and $n\in\NN$.

\begin{definition}\label{defSn+}
\noindent
\begin{enumerate}[(a)]
\item A square matrix $u\in{M_n}(A)$ is called \emph{magic} if all its entries are projections (self-adjoint idempotents: $p=p^*=p^2$) and each row and column sums up to $\I$.
\item Let us denote by $\Pol(S_n^+)$ the unital $*$-algebra generated by $n^2$ elements $u_{ij}$ ($1\leq{i,j}\leq{n}$) with the relations
\[
u_{jk}^*=u_{jk}=u_{jk}^2,\qquad\forall\;1\leq{j,k}\leq{n},
\]
and
\[
\sum_{j=1}^nu_{jk}=\I=\sum_{j=1}^nu_{kj},\qquad\forall\;1\leq{k}\leq{n}.
\]
\item The \emph{free permutation group} $S_n^+$ is defined so that $\C(S_n^+)$ is the universal \cst-algebra generated by the entries of an $n\times{n}$ magic square matrix $u=(u_{jk})$, i.e.~the completion of $Pol(S_n^+)$ with respect to the (semi-)norm
\[
\|c\|=\sup_\rho\bigl\|\rho(c)\bigr\|,
\]
where the supremum is taken over all $*$-representations of $\Pol(S_n^+)$ on some Hilbert space (prove that this sup is finite!). It is a compact quantum
group with coproduct
\[
\Delta\colon\C(S_n^+)\longrightarrow\C(S_n^+)\tens_\textrm{\tiny{min}}\C(S_n^+)
\]
determined by $\Delta(u_{jk})=\sum\limits_{\ell=1}^n u_{j\ell}\tens{u}_{\ell{k}}$.
\end{enumerate}
Note that the relations $u_{jk}u_{j\ell}=\delta_{k\ell}u_{jk}$, $u_{kj}u_{\ell j}=\delta_{k\ell}u_{kj}$ are automatically also satisfied for any magic unitary, since projections whose sum is a projection must be mutually orthogonal.

\end{definition}

Let us remark that other completions of $\Pol(S_n^+)$ yielding compact quantum groups may also be considered.

We have
\begin{itemize}
\item for $n=1,2,3$ the \cst-algebra $\C(S_n^+)$ is commutative and $\C(S_n^+)\cong\C(S_n)$, i.e.~$S_n^+$ is isomorphic to the permutation group $S_n$.
\item For $n\geq{4}$ $\C(S_n^+)$ is noncommutative and $\dim\C(S_n^+)=+\infty$, i.e.~there exist (infinitely many) genuine ``quantum permutations''. E.g.
\[
\begin{bmatrix}
\I-p&p&0&0\\
p&\I-p&0&0\\
0&0&\I-q&q\\
0&0&q&\I-q
\end{bmatrix}
\]
with $p,q$ two arbitrary projections.
\end{itemize}

\begin{remark}
$S^+_n$ is a also called a \emph{liberation} of $S_n$, since we ``freed'' the functions on the permutation group from their commutativity constraint.
\end{remark}

\subsection{$\mathrm{SU}_q(2)$}

For $q\in\RR\setminus\{0\}$ the universal \cst-algebra generated by $\alpha,\gamma$ with relations
\[
\begin{array}{c}
\alpha^*\alpha+\gamma^*\gamma=\I,\qquad\alpha\alpha^*+q^2\gamma^*\gamma=\I,\\
\gamma^*\gamma=\gamma\gamma^*,\qquad\alpha\gamma=q\gamma\alpha,\qquad\alpha\gamma^*=q\gamma^*\alpha
\end{array}
\]
can be turned into a compact quantum group, with the comultiplication defined so that
\[
\Delta
\begin{bmatrix}
\alpha&-q\gamma^*\\
\gamma&\alpha^*
\end{bmatrix}
=
\begin{bmatrix}
\alpha&-q\gamma^*\\
\gamma&\alpha^*
\end{bmatrix}
\tens
\begin{bmatrix}
\alpha&-q\gamma^*\\
\gamma&\alpha^*
\end{bmatrix},
\]
i.e.~$\Delta(\alpha)=\alpha\tens\alpha-q\gamma^*\tens\gamma$, etc.

For $q=1$ we have $\C\bigl(\mathrm{SU}_1(2)\bigr)=\C\bigl(\mathrm{SU}(2)\bigr)$ -- the \cst-algebra of continuous functions on the special unitary group $\mathrm{SU}(2)$.

\subsection{Further reading}

Information contained in the first four sections can be found in several books on Hopf algebras or on the algebraic approach to quantum groups. As examples we would like to mention \cite{Sweedler}, \cite{abe}, \cite{Montgomery}, \cite{ks} and \cite{timm}. A standard source of introductory information on \cst-algebras is the book \cite{Murphy}. The notion of compact quantum groups as defined in Section \ref{cqgs} was introduced by S.L.~Woronowicz in \cite{cqg}; earlier the special case of compact matrix quantum groups was established in \cite{pseudogr}. Free permutation groups were introduced in \cite{Wangsym}; for a survey describing their properties we refer to \cite{BBC}.

\section{From analysis to algebra and back, via representations}

In this chapter we formalize the correspondence between the algebraic theory of $*$-Hopf algebras and analytic theory of compact quantum groups introduced in the first chapter. We also present some objects naturally associated to compact quantum groups and define certain properties of the latter.

\subsection{Compact quantum groups}

The following definition restates Definition \ref{defQG1} in a slightly different language.

\begin{definition}\label{DefCQG}\label{defQG2}
An \emph{algebra of functions on a compact quantum group} is a unital \cst-algebra $\sA$ with a unital $*$-algebra homomorphism $\Delta\colon\sA\to\sA\tensmin\sA$ such that
\[
(\id\tens\Delta)\comp\Delta=(\Delta\tens\id)\comp\Delta\tag{coassociativity}
\]
and the \emph{cancellation rules} hold:
\[
\Lin\bigl\{\Delta(a)(b\tens\I)\st{a,b}\in\sA\bigr\}=\Lin\bigl\{(a\tens\I)\Delta(b)\st{a,b}\in\sA\bigr\}=\sA\tensmin\sA.
\]
\end{definition}

We will write $\sA=\C(\GG)$ and call $\GG$ the compact quantum group. 

\subsection{Convolution of probability measures on a compact group}

Let $G$ be a compact group. We will identify finite (complex) measures on $G$ with continuous functionals on $\C(G)$. Thus given two finite measures $\mu$ and $\nu$ on $G$, their convolution $\mu\star\nu$ is defined by
\[
\int\limits_Gf(g)\,d_{\mu\star\nu}(g)=\int\limits_G\int\limits_Gf(g_1g_2)\,d_\mu(g_1)d_\nu(g_2)
\]
for all $f\in\C(G)$. Note that convolution of probability measures remains a probability measure.

The Haar measure on $G$ is the unique bi-invariant measure $\mu_h\in\Prob(G)$ (where $\Prob(G)$ denotes the set of all probability measures on $G$) such that for any $g\in{G}$ and any Borel set $A\subset{G}$
\[
\mu_h(gA)=\mu_h(Ag)=\mu_h(A).
\]
In other words, it is the unique measure such that
\[
\nu\star\mu_h=\mu_h\star\nu=\mu_h,\qquad\nu\in\Prob(G).
\]

\begin{definition}
Let $\GG$ be a compact quantum group. Given two functionals $\ph,\psi\in\C(\GG)^*$ their \emph{convolution} is defined by
\[
\ph\star\psi=(\ph\tens\psi)\comp\Delta.
\]
\end{definition}

Note that the above is in fact a special case of Definition \ref{Defconv}.
Convolution of states (normalised positive functionals) is a state. We view states on $\C(\GG)$ as probability measures on $\GG$ (and sometimes write simply $\Prob(\GG)$ for the set of states on $\C(\GG)$).

\subsection{Haar state}

\begin{definition}
A state $\bh\in\Prob(\GG)$ is called a \emph{Haar state} if for all $a\in\C(\GG)$
\[
(\bh\tens\id)\bigl(\Delta(a)\bigr)=(\id\tens{\bh})\bigl(\Delta(a)\bigr)=\bh(a)\I;
\]
equivalently for each $\mu\in\C(\GG)^*$
\[
\bh\star\mu=\mu\star{\bh}=\mu(\I)\bh;
\]
equivalently for each $\omega\in\Prob(\GG)$
\[
\bh\star\omega=\omega\star{\bh}=\bh.
\]
\end{definition}

\begin{theorem}[\cite{cqg}]
Every compact quantum group has a unique Haar state $\bh$.
\end{theorem}

The proof uses cancellation laws. A rough idea is to take a faithful state $\omega\in\Prob(\GG)$ and show that
\[
\bh=\lim_{n\to\infty}\tfrac{1}{n}\sum_{k=1}^n\omega^{\star{k}}.
\]

When $G$ is a compact group then the Haar state on $\C(G)$ is given by integration with respect to the Haar measure, while for $\GG=\hh{\Gamma}$ for a discrete group $\Gamma$, the Haar state is
\[
\bh\biggl(\sum_\gamma{c_\gamma}\lambda_\gamma\biggr)=c_e.
\]

\subsection{Representations}

A (finite-dimensional, unitary, continuous) representation of a compact group $G$ is a continuous map $U\colon{G}\to\mathrm{U}(n)$ (where $n \in \NN$ and $\mathrm{U}(n)$ denotes $n$ by $n$ complex unitary matrices) such that
\[
U(gh)=U(g)U(h),\qquad{g,h}\in{G}.
\]
Looking at matrix entries we can view the representation $U$  as a single element $U\in{M_n}\bigl(\C(G)\bigr)$.
\begin{definition}
A \emph{finite-dimensional unitary, continuous representation} of a compact quantum group $\GG$ is a unitary $U=[u_{ij}]_{i,j=1,\dotsc,n}\in{M_n}\bigl(\C(\GG)\bigr)$ such that
\[
\Delta(u_{ij})=\sum_{k=1}^n{u_{ik}}\tens{u_{kj}},\qquad{i,j}=1,\dotsc,n.
\]
Equivalently, identifying $M_n\bigl(\C(\QG)\bigr)$ with $M_n\tens\C(\QG)$ we can write the above formula as
\begin{equation}\label{Deluij}
(id\tens\Delta)(U)=U_{12}U_{13}.
\end{equation}
\end{definition}

Linear combinations of $\{u_{ij}\}$ are called the \emph{coefficients} of $U$. A slightly more general notion of a \emph{non-degenerate} representation corresponds to \emph{invertible} element $U\in{M_n}\bigl(\C(\QG)\bigr)$ satisfying \eqref{Deluij}.

We say that two representation $U$ and $V$ of $\GG$ are \emph{equivalent} (written $U\approx{V}$) if there is a $y\in\mathrm{GL}(n)$ such that
\[
(y\tens\I)U=V(y\tens\I).
\]
\begin{fact}
Any non-degenerate representation is equivalent to a unitary one.
\end{fact}

\subsection{Operations on representations}

The class of all finite dimensional representations of $\GG$ will be denoted by $\Rep(\GG)$. Let $U\in{M_n}\bigl(\C(\GG)\bigr)$ and $V\in{M_m}\bigl(\C(\GG)\bigr)$ be elements of $\Rep(\QG)$. We have

\begin{itemize}
\item the \emph{direct sum} $U\oplus{V}\in{M_{n+m}}\bigl(\C(\GG)\bigr)$ of $U$ and $V$ defined as
\[
U\oplus{V}=
\begin{bmatrix}
U&0_{n,m}\\
0_{m,n}&V
\end{bmatrix},
\]
\item the \emph{tensor product} of $U$ and $V$ defined as the matrix $U\tens{V}\in{M_{nm}}\bigl(\C(\GG)\bigr)$ with entries
\[
[U\tens{V}]_{(i,j)(k,l)}=u_{ij}v_{kl},\qquad{i,j=1,\dotsc,n,}\;k,l=1,\dotsc,m.
\]
\end{itemize}

Also for $U\in\Rep(\GG)$ the \emph{adjoint} of $U$ is $\overline{U}\in{M_n}\bigl(\C(\GG)\bigr)$
\[
\overline{U}=[u_{ij}^*]_{i,j=1,\dotsc,n}.
\]

We say that $U$ is \emph{irreducible} if it cannot be non-trivially decomposed as a direct sum of other representations. Equivalently, if there is no non-trivial projection $p\in{M_n}$ such that
\[
(p\tens\I)U=U(p\tens\I).
\]
We let $\Irr(\GG)$ denote the set of all (equivalence classes) of irreducible representations of $\GG$ and for each $\alpha\in\Irr(\GG)$ we choose a unitary representative
\[
U^\alpha=
\begin{bmatrix}
u^\alpha_{11}&\dotsc&u^\alpha_{1n_\alpha}\\
\vdots&\ddots&\vdots\\
u^\alpha_{n_\alpha1}&\dotsc&u^\alpha_{n_\alpha{n_\alpha}}
\end{bmatrix}\in{M_{n_\alpha}}\bigl(\C(\GG)\bigr).
\]

\begin{theorem}
Any unitary representation of $\GG$ decomposes as a direct sum of irreducible ones. The set of coefficients of all finite-dimensional unitary (equivalently non-degenerate) representations of $\GG$ forms a unital dense $*$-subalgebra of $\C(\GG)$, denoted $\Pol(\GG)$. The set $\bigl\{u_{ij}^\alpha\st\alpha\in\Irr(\GG),\:i,j=1,\dotsc,n_\alpha\bigr\}$ is a linear basis of $\Pol(\GG)$. With $\eps\colon\Pol(\GG)\to\CC$ and $S\colon\Pol(\GG)\to\Pol(\GG)$ defined by
\[
\eps(u^\alpha_{i,j})=\delta_{ij},\qquad{S}(u^\alpha_{ij})={u^\alpha_{ij}}^*
\]
the $*$-algebra $\Pol(\GG)$ becomes a Hopf $*$-algebra.
\end{theorem}

It is important to note here that neither $\eps$ nor $S$ need to extend to the \cst-algebra $\C(\GG)$. Finally let us mention that $U\in\Rep(\GG)$ is called \emph{fundamental} if its coefficients generate $\C(\GG)$ as a \cst-algebra. An instance of such a situation can be seen in Definition \ref{defSn+}.

\subsection{Orthogonality}

\begin{theorem}\label{thmQalpha}
The Haar state is faithful on $\Pol(\GG)$ (i.e.~for $a\in\Pol(\QG)$ if $\bh(a^*a)=0$ then $a=0$). For each $\alpha\in\Irr(\GG)$ there exists a unique positive matrix $Q_\alpha\in\mathrm{GL}(n_\alpha)$ such that $\Tr(Q_\alpha)=\Tr(Q_\alpha^{-1})\geq{n_\alpha}$ and denoting $\Tr(Q_\alpha)$ by $d_\alpha$ we have for all $\alpha,\beta\in\Irr(\GG)$, $i,j\in\{1,\dotsc,n_\alpha\}$ and $k,l\in\{1,\dotsc,n_\beta\}$
\[
\begin{split}
\bh\bigl(u^\alpha_{ij}(u^\beta_{kl})^*\bigr)&=\delta_{\alpha\beta}\delta_{ik}\frac{[Q_\alpha]_{lj}}{d_\alpha},\\
\bh\bigl((u^\alpha_{ij})^*u^\beta_{kl}\bigr)&=\delta_{\alpha\beta}\delta_{jl}\frac{[Q_\alpha^{-1}]_{ki}}{d_\alpha}.
\end{split}
\]
\end{theorem}

The matrices $\{Q_\alpha\}_{\alpha \in \Irr(\GG)}$ have various incarnations:
\begin{itemize}
\item as so-called \emph{Woronowicz characters} on $\Pol(\GG)$;
\item generators of the \emph{scaling automorphism group} $\tau$;
\item witnesses of non-traciality of $h$;
\item witnesses of unboundedness of $S$.
\end{itemize}
Note that by changing the basis (i.e.\ passing to an equivalent unitary representation) one can always assume that a given matrix $Q_{\alpha}$ is diagonal. Note also that these matrices are sometimes denoted $F_\alpha$ (e.g.~in \cite{pseudogr}).

\subsection{Kac property}

\begin{definition}\label{thmKac}
A compact quantum group $\GG$ is of \emph{Kac type} if  $Q_\alpha=\I$ for all $\alpha \in \Irr(\GG)$; equivalently, $S^2=\id_{\Pol(\GG)}$; equivalently $\bh$ is a trace; equivalently the ‘quantum dimensions’ $d_\alpha$ are equal to $n_\alpha$.
\end{definition}

\subsection{From $\Pol(\GG)$ to $\C(\GG)$}

In order to obtain $\C(\GG)$ from $\Pol(\GG)$ we need ‘good’ \cst-norms on $\Pol(\GG)$. Examples of such norms are
\begin{itemize}
\item the \emph{universal norm}
\[
\|a\|_\textrm{\tiny{u}}=\sup\bigl\{\|\pi(a)\|\st\pi\colon\Pol(\GG)\to\B(\cH)\text{ is a unital $*$-homomorphism}\bigr\}
\]
(completion of $\Pol(\GG)$ in this norm, denoted by $\C_\textrm{\tiny{u}}(\GG)$, admits good $\Delta_\textrm{\tiny{u}}$, $\bh_\textrm{\tiny{u}}$ etc.)
\item the \emph{reduced norm}
\[
\|a\|_\textrm{\tiny{r}}=\bigl\|\pi_h(a)\bigr\|,
\]
where $\pi_h$ is the GNS representation of $\Pol(\GG)$ defined by the Haar state (completion of $\Pol(\GG)$ in this norm, denoted by $\C_\textrm{\tiny{r}}(\GG)$, admits good $\Delta_\textrm{\tiny{r}}$, $\bh_\textrm{\tiny{r}}$ etc.).
\end{itemize}

Of course $\|\,\cdot\,\|_\textrm{\tiny{r}}\leq\|\,\cdot\,\|_\textrm{\tiny{u}}$. The following definition contains several equivalent characterizations -- these are not elementary to show!

\begin{definition}
A compact quantum group $\GG$ is \emph{coamenable} if $\|\,\cdot\,\|_\textrm{\tiny{r}}=\|\,\cdot\,\|_\textrm{\tiny{u}}$; equivalently, $\bh_\textrm{\tiny{u}}(\GG)$ is faithful on $\C_\textrm{\tiny{u}}(\GG)$; equivalently, $\eps$ extends to a character on $\C_\textrm{\tiny{r}}(\GG)$.
\end{definition}

Finally we return to the notion of the morphism between compact quantum groups, indicated already in the first chapter. It turns out that it is best formulated on the universal level.

\begin{definition}\label{defmor2}
A morphism between compact quantum groups $\GG_1$ and $\GG_2$ is a unital $*$-homomorphism
$\pi\colon{\C_\textrm{\tiny{u}}(\GG_2)}\to{\C_\textrm{\tiny{u}}(\GG_1)}$ such that
\[
\Delta_1\comp\pi=(\pi\tens\pi)\comp\Delta_2.
\]
\end{definition}
In fact the definition above can be formulated algebraically: there is a natural one-to-one correspondence between  unital $*$-homomorphisms $\pi\colon{\C_\textrm{\tiny{u}}(\GG_2)}\to{\C_\textrm{\tiny{u}}(\GG_1)}$ and $\rho\colon \Pol(\GG_2)\to\Pol(\GG_1) $ intertwining the respective coproducts.

\subsection{Further reading}

As stated above, the existence of the Haar state was first announced in \cite{su2} with proof given in \cite{pseudogr} and perfected in \cite{cqg} under the assumption of $\C(\QG)$ being separable, and in \cite{vdH} in the general case. Most of the contents of this chapter, together with detailed proofs, can be found in \cite{cqg} and \cite{mvd}. The notions of `good' \cst-norms in context of the Hopf algebras associated to quantum groups was studied for example in \cite{KyedSoltan}.  For more information on coamenability of compact quantum groups we refer to \cite{BMT}.

\section{Introduction to discrete quantum groups}
In this chapter we show how to each compact quantum group one can associate a `discrete quantum group', playing in a sense the role of the dual object.

\subsection{Irreducible representations revisited and the Fourier transform}

We begin by recalling certain facts from the last chapter.
Let $\GG$ be again a  compact quantum group defined as the ``virtual object'' corresponding to a unital \cst-algebra denoted $\C(\GG)$ equipped with a comultiplication (i.e.~ a coassociative unital $*$-homomorphism) $\Delta_\GG$ satisfying appropriate density conditions. Let $\Irr(\GG)$ denote the set of equivalence classes of irreducible representations of $\GG$. For each $\alpha\in\Irr(\GG)$ we choose a unitary representative $U^\alpha$ of $\alpha$, so
\[
U^\alpha\in\B(\cH_\alpha)\tens(\GG),
\]
where $\cH_\alpha$ is a finite dimensional Hilbert space. Denote the dimension of $\cH_\alpha$ by $n_\alpha$. Upon choosing an orthonormal basis in $\cH_\alpha$ we can express $U^\alpha$ as
\[
U^\alpha=\sum_{i,j=1}^{n_\alpha}e_{i,j}\tens{u^\alpha_{i,j}},
\]
where $\{e_{i,j}\}_{i,j=1,\dotsc,n_\alpha}$ is the corresponding basis of matrix units in $\B(\cH_\alpha)$ (we use here the notation $u_{i,j}$ as opposed to $u_{ij}$ due to the leg numbering notation to be introduced below). In other words $U^\alpha$ becomes a unitary matrix
\[
U^\alpha=
\begin{bmatrix}
u^\alpha_{1,1}&\dotsc&u^\alpha_{1,n_\alpha}\\
\vdots&\ddots&\vdots\\
u^\alpha_{n_\alpha,1}&\dotsc&u^\alpha_{n_\alpha,n_\alpha}
\end{bmatrix}.
\]

The fact that each $U^\alpha$ is a representation of $\GG$ can be expressed either by saying that
\[
\Delta_\GG(u^\alpha_{i,j})=\sum_{k=1}^{n_\alpha}u^\alpha_{i,k}\tens{u^\alpha_{k,j}},\qquad{i,j=1,\dotsc,n_\alpha}
\]
or, using the \emph{leg numbering notation}, that
\[
(\id\tens\Delta_\GG)U^\alpha=U^\alpha_{12}U^\alpha_{13},
\]
where
\[
U^\alpha_{12}=U\tens\I\in\B(\cH_{\alpha})\tens\C(\GG)\tensmin\C(\GG)
\]
and
\[
U^\alpha_{13}=\sum_{i,j=1}^{n_\alpha}e_{i,j}\tens\I\tens{u^\alpha_{i,j}}\in\B(\cH_{\alpha})\tens\C(\GG)\tensmin\C(\GG)
\]
(cf.~Section \ref{algCoalg}).

Recall from the previous chapters that $\Pol(\GG)$ defined as the linear span of the set
\[
\bigl\{u^\alpha_{i,j}\st\alpha\in\Irr(\GG),\:i,j=1,\dotsc,n_\alpha\bigr\}
\]
is a dense unital $*$-subalgebra of $\C(\GG)$ and, moreover, with comultiplication inherited from $\C(\GG)$, the $*$-algebra $\Pol(\GG)$ becomes a Hopf $*$-algebra. The antipode of $\Pol(\GG)$ will be denoted by $S$.

\subsection{Haar measure and irreducible representations}

Let $\bh$ be the Haar measure (Haar state) of $\GG$.

\begin{proposition}\label{pierwsza}
For any $\alpha\in\Irr(\GG)$ there exists a unique $Q_\alpha\in\B(\cH_\alpha)$ such that
\begin{itemize}
\item $Q_\alpha$ is invertible,
\item $(Q_\alpha\tens\I)U^\alpha=\bigl((\id\tens{S^2})U^\alpha\bigr)(Q_\alpha\tens\I)$,
\item $\Tr(Q_\alpha)=\Tr(Q_\alpha^{-1})>0$.
\end{itemize}
Moreover the operator $Q_\alpha$ is positive.
\end{proposition}

The proof of Proposition \ref{pierwsza} can be found in \cite{pseudogr}. Recall the \emph{Peter-Weyl-Woronowicz orthogonality relations} satisfied by the matrices $Q_{\alpha}$ displayed in Theorem \ref{thmQalpha}.
%

\subsection{Building the dual of $\GG$}

Let $\hh{\sA}$ be the \cst-algebra defined as the $\cc0$-direct sum
\[
\hh{\sA}=\bigoplus_{\alpha\in\Irr(\GG)}\B(\cH_\alpha).
\]
Unless we are in the basic situation when $\GG$ is \emph{finite} (i.e.~$\dim\C(\GG)<+\infty$), the \cst-algebra $\hh{\sA}$ does not have a unit. Therefore we will be forced to deal with the \emph{multiplier algebra} $\M(\hh{\sA})$ of $\hh{\sA}$. Due to the relatively simple structure of $\hh{\sA}$, the multiplier algebra has a very convenient description: $\M(\hh{\sA})$ is the $\ell^\infty$-direct sum of the finite dimensional blocks $\B(\cH_\alpha)$.

Now let
\[
\ww=\bigoplus_{\alpha\in\Irr(\GG)}U^\alpha.
\]
It is not hard to see that $\ww$ is a unitary element of $\M\bigl(\hh{\sA}\tensmin\C(\GG)\bigr)$ and (suppressing the apparent difficulties with the rigorous interpretation of this formula) we have
\begin{equation}\label{www}
(\id\tens\Delta_\GG)\ww=\ww_{12}\ww_{13},
\end{equation}
where again we used the leg numbering notation. Unitarity of $\ww$ and \eqref{www} mean that $\ww$ is an infinite dimensional representation of $\GG$. It is a fact that such representations also decompose into direct sums of irreducible ones in an appropriate sense.

Finally let us define a $*$-subalgebra $\hcA$ of $\hh{\sA}$ as the algebraic direct sum
\[
\hcA=\bigoplus_{\alpha\in\Irr(\GG)}\B(\cH_\alpha).
\]
This means that each element of $\hcA$ has only finitely many non-zero components in the direct summands of $\hh{\sA}$. The algebra $\hh{\sA}$ will play the role of the algebra of continuous functions vanishing at infinity on the dual of $\GG$, while $\hcA$ will correspond to compactly supported functions. It is easy to see that $\hcA$ is dense in $\hh{\sA}$.

\subsection{Fourier transform}\label{fSect}

Let us define a linear map $\cF\colon\Pol(\GG)\to\hcA$ by
\begin{equation}\label{F}
\cF(a)=(\id\tens\bh)\bigl((\I\tens{a})\ww^*\bigr), \;\;\; a \in \Pol(\GG).
\end{equation}
It turns out that $\cF$ is an isomorphism of vector spaces with inverse given by
\begin{equation}\label{Finv}
\cF^{-1}(x)=(\hh{\bh}_{\textrm{\tiny{L}}}\tens\id)\bigl((x\tens\I)\ww\bigr),\qquad{x}\in\hcA,
\end{equation}
where $\hh{\bh}_{\textrm{\tiny{L}}}$ is the linear functional on $\hcA$ defined as
\[
\hh{\bh}_{\textrm{\tiny{L}}}(x)=\sum_{\beta\in\Irr(\GG)}\Tr(Q_\beta)\Tr(Q_\beta^{-1}x_\beta),\qquad{x}\in\hcA
\]
with $x_\beta$ denoting the component of $x$ in the block $\B(\cH_\beta)\subset\hcA$ (in particular the above sum is finite).

The proof that \eqref{F} and \eqref{Finv} are mutually inverse to one another involves orthogonality relations for matrix elements of irreducible representations $\{U^\alpha\}_{\alpha\in\Irr(\GG)}$ (see \cite[Section 2]{PodSLW}).

\subsection{Main theorem}

\begin{theorem}[{\cite[Theorem 2.1]{PodSLW}}]\label{main}
Let $\sB$ be a \cst-algebra and $\vv\in\M\bigl(\sB\tensmin\C(\GG)\bigr)$ a unitary such that
\[
(\id\tens\Delta_\GG)\vv=\vv_{12}\vv_{13}.
\]
Then there exists a unique \emph{non-degenerate} $^*$-homomorphism  $\Phi:\hh{\sA} \to\M(\sB)$ such that
\[
(\Phi\tens\id)\ww=\vv.
\]
\end{theorem}

The formulation of Theorem \ref{main} uses the notion of a \emph{nondegenerate} $^*$-homomorphism; this condition means that the span of $\Phi(\hat{\sA})\sB$ is  dense in $\sB$. We will not go into the technical difficulties related to composing of such morphisms  and refer the reader to \cite{unbo,gen}; to simplify the notation we will simply take $\Phi\in \Mor(\hat{\sA},\sB)$ to mean that $\Phi:\hh{\sA} \to\M(\sB)$ is a non-degenerate $^*$-homomorphism.


\begin{proof}[Proof of Theorem \ref{main}]
First let us address the question of uniqueness of $\Phi$. Assume we have $\Phi\in\Mor(\hh{\sA},\sB)$ such that $(\Phi\tens\id)\ww=\vv$. Then for any $a\in\Pol(\GG)$
\[
\begin{split}
\Phi\bigl(\cF(a)\bigr)&=\Phi\bigl((\id\tens\bh)\bigl((\I\tens{a})\ww^*\bigr)\bigr)\\
&=(\id\tens\bh)\bigl((\Phi\tens\id)\bigl((\I\tens{a})\ww^*\bigr)\bigr)\\
&=(\id\tens\bh)\bigl((\I\tens{a})\vv^*\bigr)
\end{split}
\]
Writing $a=\cF^{-1}(x)$ with $x\in\hcA$ we obtain
\[
\Phi(x)=(\id\tens\bh)\bigl(\bigl(\I\tens\cF^{-1}(x)\bigr)\vv^*\bigr),\qquad{x}\in\hcA,
\]
so $\Phi$ is determined uniquely on $\hcA$ which is dense in $\hh{\sA}$.

For the proof of existence of $\Phi$ let us denote the map
\[
\Pol(\GG)\ni{a}\longmapsto(\id\tens\bh)\bigl((\I\tens{a})\vv^*\bigr)\in\M(\sB)
\]
by $\cF_\vv$ and define a linear map $\Phi\colon\hcA\to\M(\sB)$ as the composition
\[
\Phi=\cF_\vv\comp\cF^{-1}.
\]

Take $a\in\Pol(\GG)$. We will compute the expression
\[
\mathcal{X}=(\id\tens\id\tens\bh)\bigl((\id\tens\Delta_\GG)\bigl((\I\tens{a})\vv^*\bigr)\bigr)
\]
in two ways. First, using the fact that $\bh$ is the Haar measure we find that
\[
\mathcal{X}=\cF_\vv(a)\tens\I.
\]
On the other hand, using the fact that $\Delta_\GG$ is a $*$-homomorphism and formula $(\id\tens\Delta_\GG)\vv=\vv_{12}\vv_{13}$, we obtain
\[
\begin{split}
\mathcal{X}&=(\id\tens\id\tens\bh)\bigl(\bigl(\I\tens\Delta_\GG(a)\bigr)(\id\tens\Delta_\GG)\vv^*\bigr)\\
&=(\id\tens\id\tens\bh)\bigl(\bigl(\I\tens\Delta_\GG(a)\bigr)\vv^*_{13}\vv^*_{12}\bigr)\\
&=\bigl[(\id\tens\bh)\bigl(\Delta_\GG(a)_{23}\vv^*_{13}\bigr)\bigr]\vv^*.
\end{split}
\]
Since $\vv$ is unitary, we immediately see that this implies
\begin{equation}\label{thisone}
(\id\tens\bh)\bigl(\Delta_\GG(a)_{23}\vv^*_{13}\bigr)=\bigl(\cF_\vv(a)\tens\I\bigr)\vv.
\end{equation}
Now we can multiply both sides of \eqref{thisone} from the right by $\I\tens{b^*}$ (with $b\in\Pol(\GG)$) and apply $(\id\tens\bh)$ to obtain
\[
(\id\tens\bh)\bigl[\bigl((\id\tens\bh)\bigl(\Delta_\GG(a)_{23}\vv^*_{13}\bigr)\bigr)(\I\tens{b^*})\bigr]
=(\id\tens\bh)\bigl(\cF_\vv (a)_{12}\vv(\I\tens{b^*})\bigr)
\]
which can be rewritten as
\begin{equation}\label{FF1}
\begin{split}
\cF_\vv\bigl((\bh\tens\id)\bigl(\Delta(a)(b^*\tens\I)\bigr)\bigr)&=\cF_\vv(a)\cdot(\id\tens\bh)\bigl(\vv(\I\tens{b^*})\bigr)\\
&=\cF_\vv(a)\bigl[(\id\tens\bh)\bigl((\I\tens{b})\vv^*\bigr)\bigr]^*\\
&=\cF_\vv(a)\cF_\vv(b)^*
\end{split}
\end{equation}

Exactly the same calculation with $\vv$ replaced by $\ww$ yields
\begin{equation}\label{FF2}
\cF\bigl((\bh\tens\id)\bigl(\Delta(a)(b^*\tens\I)\bigr)\bigr)=\cF(a)\cF(b)^*
\end{equation}
for all $a,b\in\Pol(\GG)$. It follows that if $x=\cF(a)$ and $y=\cF(b)$ then
\begin{equation}\label{Phixy}
\begin{split}
\Phi(xy^*)&\,\,\;\!=\cF_\vv\bigl(\cF^{-1}(xy^*)\bigr)\\
&\,\,\;\!=\cF_\vv\bigl(\cF^{-1}\bigl(\cF(a)\cF(b)^*\bigr)\bigr)\\
&\overset{\textrm{\eqref{FF2}}}{=}\cF_\vv\bigl((\bh\tens\id)\bigl(\Delta(a)(b^*\tens\I)\bigr)\bigr)\\
&\overset{\textrm{\eqref{FF1}}}{=}\cF_\vv(a)\cF_\vv(b)^*\\
&\,\,\;\!=\cF_\vv\bigl(\cF^{-1}(x)\bigr)\cF_\vv\bigl(\cF^{-1}(y)\bigr)^*=\Phi(x)\Phi(y)^*.
\end{split}
\end{equation}

Now any $z\in\hcA$ can be written in the form $xy^*$ with $x,y\in\hcA$ and then $z^*=yx^*$, so
\[
\Phi(z^*)=\Phi(yx^*)=\Phi(y)\Phi(x)^*=\bigl(\Phi(x)\Phi(y)^*\bigr)^*=\Phi(z)^*,
\]
i.e.~$\Phi$ is a $*$-map. Using this and \eqref{Phixy} we obtain also multiplicativity of $\Phi$:
\[
\Phi(xy)=\Phi\bigl(x(y^*)^*\bigr)=\Phi(x)\Phi(y^*)^*=\Phi(x)\Phi(y),\qquad{x,y}\in\hcA.
\]

There are two points whose proof we will skip, namely:
\begin{itemize}
\item the map $\Phi\colon\hcA\to\M(\sB)$ is continuous and consequently it extends to a $*$-homomorphism of \cst-algebras $\hh{\sA}\to\M(\sB)$,
\item $\Phi$ is non-degenerate, i.e.~the span of elements of the form $\Phi(y)b$ with all possible $x\in\hh{\sA}$ and $b\in\sB$ is dense in $\sB$.
\end{itemize}
Both are treated in detail in \cite{PodSLW}.

Let us finish the proof by sketching an argument showing that $\Phi$ defined above does indeed satisfy
\[
(\Phi\tens\id)\ww=\vv.
\]
To that end let $\widetilde{\vv}=(\Phi\tens\id)\ww$. For any $a\in\Pol(\GG)$ we have
\[
\begin{split}
(\id\tens\bh)\bigl((\I\tens{a})\widetilde{\vv}^*\bigr)&=(\id\tens\bh)\bigl((\I\tens{a})(\Phi\tens\id)\ww^*\bigr)\\
&=\Phi\bigl((\id\tens\bh)\bigl((\I\tens{a})\ww^*\bigr)\bigr)\\
&=\Phi\bigl(\cF(a)\bigr)\\
&=\cF_\vv\bigl(\cF^{-1}\bigl(\cF(a)\bigr)\bigr)\\
&=\cF_\vv(a)=(\id\tens\bh)\bigl((\I\tens{a})\vv^*\bigr).
\end{split}
\]
This result is sufficient to conclude that $\widetilde{\vv}=\vv$, but one has to use decomposition of (infinite dimensional) unitary representations into irreducible ones and orthogonality relations or the argument given on \cite[Page 397]{PodSLW}. In any case we do get $(\Phi\tens\id)\ww=\widetilde{\vv}=\vv$ which ends the proof.
\end{proof}

\subsection{The dual quantum group}

\subsection{Comultiplication}\label{hDelta}

Let $\sB=\hh{\sA}\tensmin\hh{\sA}$ and define $\vv\in\M\bigl(\sB\tensmin\C(\GG)\bigr)$ as
\[
\vv=\ww_{23}\ww_{13}.
\]
We have
\[
\begin{split}
(\id\tens\Delta_\GG)\vv&=\bigl[(\id\tens\Delta_\GG)\ww\bigr]_{234}\bigl[(\id\tens\Delta_\GG)\ww\bigr]_{134}\\
&=\ww_{23}\ww_{24}\ww_{13}\ww_{14}\\
&=\ww_{23}\ww_{13}\ww_{24}\ww_{14}\\
&=\vv_{12}\vv_{13},
\end{split}
\]
where in the first three lines the leg numbers refer to $\hh{\sA}\tensmin\hh{\sA}\tensmin\C(\GG)\tensmin\C(\GG)$ and in the last line they refer to
$(\hh{\sA}\tensmin\hh{\sA})\tensmin\C(\GG)\tensmin\C(\GG)$ (the tensor product $\hh{\sA}\tensmin\hh{\sA}$ is treated as one leg). The element $\vv$ is unitary, so by Theorem \ref{main} there exists a unique $\hh{\Delta}\in\Mor(\hh{\sA},\hh{\sA}\tensmin\hh{\sA})$ such that
\[
(\hh{\Delta}\tens\id)\ww=\ww_{23}\ww_{13}.
\]

The morphism $\hh{\Delta}$ is coassociative. Indeed, we have
\[
\begin{split}
\bigl((\hh{\Delta}\tens\id)\comp\hh{\Delta}\tens\id\bigr)\ww&=(\hh{\Delta}\tens\id\tens\id)\ww_{23}\ww_{13}\\
&=(\hh{\Delta}\tens\id\tens\id)(\ww_{23})\cdot(\hh{\Delta}\tens\id\tens\id)(\ww_{13})\\
&=\ww_{34}\cdot\ww_{24}\ww_{14}
\end{split}
\]
and
\[
\begin{split}
\bigl((\id\tens\hh{\Delta})\comp\hh{\Delta}\tens\id\bigr)\ww&=(\id\tens\hh{\Delta}\tens\id)\ww_{23}\ww_{13}\\
&=(\id\tens\hh{\Delta}\tens\id)(\ww_{23})\cdot(\id\tens\hh{\Delta}\tens\id)(\ww_{13})\\
&=\ww_{34}\ww_{24}\cdot\ww_{14}.
\end{split}
\]
In particular
\[
\bigl((\hh{\Delta}\tens\id)\comp\hh{\Delta}\tens\id\bigr)\ww^*=\bigl((\id\tens\hh{\Delta})\comp\hh{\Delta}\tens\id\bigr)\ww^*.
\]
Multiplying both sides from the left by $(\I\tens\I\tens{a})$ (with $a\in\Pol(\GG)$) and applying $(\id\tens\id\tens\bh)$ we see that
\[
(\hh{\Delta}\tens\id)\comp\hh{\Delta}=(\id\tens\hh{\Delta})\comp\hh{\Delta}
\]
on $\hcA$ and since this algebra is dense in $\hh{\sA}$, we obtain coassociativity of $\hh{\Delta}$.

\subsection{Counit}

Setting $\sB=\CC$ and $\vv=1\tens\I\in\CC\tens\C(\GG)$ and using Theorem \ref{main} we get a unique character $\hh{e}$ of $\hh{\sA}$ with the property that
\[
(\hh{e}\tens\id)\ww=\I.
\]
A similar trick as the one leading to coassociativity of $\hh{\Delta}$ in Section \ref{hDelta} shows that
\begin{equation}\label{cou}
(\hh{e}\tens\id)\comp\Delta=\id=(\id\tens\hh{e})\comp\Delta
\end{equation}
Indeed, for the first equality we compute
\[
\begin{split}
\bigl(\bigl((\hh{e}\tens\id)\comp\Delta\bigr)\tens\id\bigr)\ww
&=(\hh{e}\tens\id\tens\id)(\ww_{23}\ww_{13})\\
&=\ww\cdot\bigl(\I\tens\bigl((\hh{e}\tens\id)\ww\bigr)\bigr)=\ww\cdot(\I\tens\I)=\ww
\end{split}
\]
and thus $(\hh{e}\tens\id)\comp\Delta=\id$ on $\hcA$ as before. The other equality of \eqref{cou} is proved in the same way.

\subsection{Haar measures}

In Section \ref{fSect} we introduced the functional $\hh{\bh}_{\textrm{\tiny{L}}}$ on $\hcA$:
\[
\hh{\bh}_{\textrm{\tiny{L}}}(x)=\sum_{\beta\in\Irr(\GG)}\Tr(Q_\beta)\Tr(Q_\beta^{-1}x_\beta),\qquad{x}\in\hcA
\]
(recall that for $x\in\hcA$ the symbol $x_\beta$ denotes the component of $x$ in the direct summand $\B(\cH_\beta)$). Similarly let
\[
\hh{\bh}_{\textrm{\tiny{R}}}(x)=\sum_{\beta\in\Irr(\GG)}\Tr(Q_\beta)\Tr(Q_\beta{x_\beta}),\qquad{x}\in\hcA.
\]
Then, although for $x\in\hcA$ the element $\hh{\Delta}(x)$ does not (usually) belong to the algebraic tensor product $\hcA\tens\hcA$, one can make sense of the expressions
\[
(\id\tens\hh{\bh}_{\textrm{\tiny{L}}})\hh{\Delta}(x),\qquad(\hh{\bh}_{\textrm{\tiny{R}}}\tens\id)\hh{\Delta}(x)
\]
and, moreover, show that we in fact have
\[
\begin{array}{r@{\;=\;}l}
(\id\tens\hh{\bh}_{\textrm{\tiny{L}}})\hh{\Delta}(x)&\hh{\bh}_{\textrm{\tiny{L}}}(x)\I_{\M(\sB)},\\
(\hh{\bh}_{\textrm{\tiny{R}}}\tens\id)\hh{\Delta}(x)&\hh{\bh}_{\textrm{\tiny{R}}}(x)\I_{\M(\sB)},
\end{array}
\qquad{x}\in\hcA.
\]
Thus $\hh{\bh}_{\textrm{\tiny{L}}}$ and $\hh{\bh}_{\textrm{\tiny{R}}}$ are respectively left and right invariant functionals on $\hcA$. They extend to so called \emph{weights} on the $\cst$-algebra $\hcA$ which have many desirable properties (e.g.~are \emph{faithful} and \emph{locally finite} as well as faithful when appropriately extended to $\M(\hh{\sA})$). In what follows we will refer to $\hh{\bh}_{\textrm{\tiny{L}}}$ and $\hh{\bh}_{\textrm{\tiny{R}}}$ as the left and right \emph{Haar measure} on the dual of $\GG$.

\subsection{Definition of a discrete quantum group}

We have by now established a number of properties of the objects
\[
\bigl(\hh{\sA},\hh{\Delta},\hh{\bh}_{\textrm{\tiny{L}}}, \hh{\bh}_{\textrm{\tiny{R}}}\bigr).
\]
One way of summarizing some of them is that they satisfy the axioms of a \emph{reduced \cst-algebraic locally compact quantum group} as defined by J.~Kustermans and S.~Vaes in \cite[Definition 4.1]{kv}. This quantum group is usually denoted by the symbol $\hh{\GG}$. Accordingly we introduce new notation:

\setlength\extrarowheight{5pt}

\begin{table}[h]
\centering
\begin{tabular}{@{\quad}c@{\quad}|@{\quad}c@{\quad}}
We write & For \\
\hline
$\cc0(\hh{\GG})$    &$\hh{\sA}$    \\
$\mathrm{c}_{00}(\hh{\GG})$&$\hcA$\\
$\Delta_{\hh{\GG}}$&$\hh{\Delta}$ \\
\end{tabular}
\end{table}

Locally compact quantum groups obtained from  compact quantum groups via the procedure described above are called \emph{discrete}. One can show that a locally compact quantum group $\GG$ is discrete if and only if the corresponding \cst-algebra is a direct sum of finite dimensional algebras (and many other equivalent conditions for a locally compact quantum group to be discrete can be given).

\subsection{(Non-)Unimodularity}

The collection $(Q_\alpha)_{\alpha\in\Irr(\GG)}$ of operators introduced in Proposition \ref{pierwsza} is not necessarily bounded (in norm), so it is not an element of neither $\cc0(\hh{\GG})$ nor $\M\bigl(\cc0(\hh{\GG})\bigr)$. It is, however, an element \emph{affiliated} with $\cc0(\hh{\GG})$ (see \cite{unbo}) which means that it can be thought of as a possibly unbounded continuous function on the quantum space $\hh{\GG}$. Let us denote this element simply by $Q$.

One can show that the following formulas hold for any $x\in\mathrm{c}_{00}(\hh{\GG})$:
\[
\begin{split}
(\hh{\bh}_{\textrm{\tiny{L}}}\tens\id)\hh{\Delta}(x)&=\hh{\bh}_{\textrm{\tiny{L}}}(x)Q^2,\\
(\id\tens\hh{\bh}_{\textrm{\tiny{R}}})\hh{\Delta}(x)&=\hh{\bh}_{\textrm{\tiny{R}}}(x)Q^{-2}.
\end{split}
\]
This means that the failure of $\hh{\bh}_{\textrm{\tiny{L}}}$ to be right invariant (as well as failure of $\hh{\bh}_{\textrm{\tiny{R}}}$ to be left invariant) is controlled by $Q^2$ ($Q^{-2}$  respectively) in much the same way as the modular function of a locally compact group controls similar phenomena for a non-unimodular group. For that reason $Q$ is called the \emph{modular element} for $\hh{\GG}$. Since there are compact quantum groups with non-trivial matrices $(Q_\alpha)_{\alpha\in\Irr(\GG)}$, we see that discrete quantum groups may be non-unimodular.

Let us finish with the following theorem which can be found in \cite{cqg,PodSLW}. It provides an extended list of conditions stated before in Definition \ref{thmKac}.

\begin{theorem}
Let $\GG$ be a compact quantum group. Then the following statements are equivalent:
\begin{enumerate}
\item the Haar measure of $\GG$ is a trace,
\item $\hh{\bh}_{\textrm{\tiny{L}}}=\hh{\bh}_{\textrm{\tiny{R}}}$ on $\hh{\GG}$, i.e.~$\hh{\GG}$ is unimodular,
\item $Q=\I$,
\item the antipode $S$ of $\GG$ satisfies $S^2=\id$,
\item the antipode $S$ is bounded.
\end{enumerate}
\end{theorem}

\subsection{Further reading}

The theory of discrete quantum groups defined as dual objects of compact quantum groups was originally developed in \cite[Section 3]{PodSLW}. Later abstract approaches bypassing the theory of compact quantum groups were used in \cite{ER94} and \cite{VD96}. There are also characterizations of discrete quantum groups among locally compact quantum groups (\cite{kv}) in terms of ideals in second duals of associated \cst-algebras (see \cite{runde}). General locally compact quantum groups have been introduced in \cite{kv} building on the work on multiplicative unitaries \cite{BS,mu} and on Kac algebras \cite{ES}. The central theme of all these developments is duality generalizing Pontriagin duality for locally compact abelian groups (cf.~\cite{pseu}).

\subsection*{Acknowledgments}
AS  was partially supported by the NCN (National Centre of Science) grant
2014/14/E/ST1/00525. UF and AS acknowledge support by the French MAEDI and MENESR and by the Polish MNiSW through the Polonium programme.


\begin{thebibliography}{66}

\bibitem{abe}
E.~Abe,
\emph{Hopf algebras},
Cambridge Tracts in Mathematics \textbf{74}, Cambridge University Press, Cambridge-New York, 1980.

\bibitem{BS}
S.~Baaj \& G.~Skandalis,
Unitaires multiplicatifs et dualit\'e pour les produits crois\'es de $\mathrm{C}^*$-alg\`ebres,
\emph{Ann.~Scient.~\'Ec.~Norm.~Sup.}, $4^{\text{\tiny e}}$ s\'erie, t.~\textbf{26} (1993), 425--488.

\bibitem{BBC} T.~Banica, J.~Bichon \& B.~Collins,
Quantum permutation groups: a survey,
\emph{Banach Center Publ.}, \textbf{78}(2007), 13--34.

\bibitem{BMT}
E.~Bedos, G.~Murphy \& L.~Tuset,
Co-amenability for compact quantum groups,
\emph{J.~Geom.~Phys.} \textbf{40} (2001), 130--153.

\bibitem{ER94}
E.G.~Effros \& Z.-J.~Ruan,
Discrete quantum groups. I. The Haar measure,
\emph{Internat.~J.~Math.} \textbf{5} (1994), 681--723.

\bibitem{ES}
M.~Enock \& J.-M.~Schwartz,
\emph{Kac algebras and duality of locally compact groups},
Springer-Verlag, Berlin, 1992.

\bibitem{UweRolf}
U.~Franz \& R.~Gohm,
Random Walks on Finite Quantum Groups,
in \emph{Quantum Independent Increment Processes, Vol.~{\rm{II}}: Structure of Quantum L\'evy Processes, Classical Probability and Physics},
U.~Franz \& M.~Sch\"urmann eds.,
Lecture Notes in Mathematics \textbf{1866}, Springer-Verlag, Heidelberg 2006.

\bibitem{ks}
A.~Klimyk, \& K.~Schm\"udgen,
\emph{Quantum Groups and Their Representations},
Texts and Monographs in Physics, 1997.

\bibitem{KyedSoltan}
D.~Kyed \& P.M.~So{\l}tan,
Property $(\mathrm{T})$ and exotic quantum group norms,
\emph{J.~Noncommut.~Geom.} \textbf{6} (2012), 773--800.

\bibitem{kv}
J.~Kustermans \& S.~Vaes,
Locally compact quantum groups,
\emph{Ann.~Scient.~\'{E}c.~Norm.~Sup.~(4)} \textbf{33} (2000), 837--934.

\bibitem{mvd}
A.~Maes \& A.~Van Daele,
Notes on compact quantum groups,
\emph{Nieuw Arch.~Wisk.~(4)} \textbf{16} (1998), 73--112.

\bibitem{Murphy}
G.J.~Murphy,
\emph{$\mathrm{C}^*$-algebras and operator theory},
Academic Press, Inc., Boston, MA, 1990.

\bibitem{NeTu}
S.~Neshveyev \& L.~Tuset,
\emph{Compact quantum groups and their representation categories},
Soci\'et\'e Math\'ematique de France 2013.

\bibitem{PodSLW}
P.~Podle\'s \& S.L.~Woronowicz,
Quantum deformation of Lorentz group,
\emph{Commun.~Math.~Phys.} \textbf{130} (1991), 381--431.

\bibitem{Montgomery}
S.~Montgomery
\emph{Hopf algebras and their actions on rings},
CBMS Regional Conference Series in Mathematics, 82. American Mathematical Society, Providence, RI, 1993.

\bibitem{runde}
V.~Runde,
Characterizations of compact and discrete quantum groups through second duals,
\emph{J.~Op.~Theory} \textbf{60} (2008), 415-428.

\bibitem{Sweedler}
M.E.~Sweedler,
\emph{Hopf algebras},
Mathematics Lecture Note Series W.A.~Benjamin, Inc., New York 1969

\bibitem{timm}
T.~Timmermann,
\emph{An Invitation to Quantum Groups and Duality},
European Mathematical Society Publishing House, 2008.

\bibitem{vdH}
A.~Van Daele,
The Haar measure on a compact quantum group,
\emph{Proc.~Amer.~Math.~Soc.} \textbf{123} (1995), 3125--3128.

\bibitem{VD96}
A.~Van Daele,
Discrete quantum groups,
\emph{J.~Algebra} \textbf{180} (1996), 431--444.

\bibitem{Wangsym}
S.~Wang,
Quantum symmetry groups of finite spaces,
\emph{Comm.~Math.~Phys.} \textbf{195} (1998), 195--211.

\bibitem{pseu}
S.L.~Woronowicz,
Pseudogroups, pseudospaces and Pontryagin duality,
in \emph{Proceedings of the International Conference on Mathematical Physics, Lausanne 1979} Lecture Notes in Physics, \textbf{116}, pp.~407--412.

\bibitem{su2}
S.L.~Woronowicz,
Twisted SU(2) group. An example of a noncommutative differential calculus.
\emph{Publ.~Res.~Inst. Math.~Sci.} \textbf{23} (1987), 117--181.

\bibitem{pseudogr}
S.L.~Woronowicz,
Compact matrix pseudogroups,
\emph{Comm.~Math.~Phys.} \textbf{111} (1987), 613--665.

\bibitem{unbo}
S.L.~Woronowicz,
Unbounded elements affiliated with $\mathrm{C}^*$-algebras and non-compact quantum groups,
\emph{Commun.~Math.~Phys.} \textbf{136} (1991), 399--432.

\bibitem{mu}
S.L.~Woronowicz,
From multiplicative unitaries to quantum groups,
\emph{Int.~J.~Math.} \textbf{7}, (1996), 127--149.

\bibitem{gen}
S.L.~Woronowicz,
$\mathrm{C}^*$-algebras generated by unbounded elements,
\\emph{Rev.~Math.~Phys.} \textbf{7}, (1995), 481--521.

\bibitem{cqg}
S.L.~Woronowicz,
Compact quantum groups,
in \emph{Sym\'etries quantiques, les Houches, Session LXIV 1995},
Elsevier 1998, pp.~845--884.

\end{thebibliography}
\end{document}